\documentclass[oneside,english]{amsart}
\usepackage{mathpazo}
\usepackage[T1]{fontenc}
\usepackage[latin9]{inputenc}
\usepackage{geometry}
\usepackage{color}
\usepackage{babel}
\usepackage{amsbsy}
\usepackage{amstext}
\usepackage{amsthm}
\usepackage{amssymb}
\usepackage{wasysym}
\usepackage[unicode=true,pdfusetitle,
 bookmarks=true,bookmarksnumbered=false,bookmarksopen=false,
 breaklinks=false,pdfborder={0 0 1},backref=false,colorlinks=false]
 {hyperref}

\makeatletter

\providecommand{\tabularnewline}{\\}

\numberwithin{equation}{section}
\numberwithin{figure}{section}
\theoremstyle{plain}
\newtheorem{thm}{\protect\theoremname}
\theoremstyle{plain}
\newtheorem{rem}[thm]{\protect\remarkname}
\theoremstyle{plain}
\newtheorem{lem}[thm]{\protect\lemmaname}
\theoremstyle{definition}
\newtheorem{defn}[thm]{\protect\definitionname}
\theoremstyle{plain}
\newtheorem{cor}[thm]{\protect\corollaryname}
\theoremstyle{plain}
\newtheorem{prop}[thm]{\protect\propositionname}
\theoremstyle{definition}
\newtheorem{problem}[thm]{\protect\problemname}
\theoremstyle{definition}
\newtheorem{example}[thm]{\protect\examplename}
\theoremstyle{remark}
\newtheorem*{claim*}{\protect\claimname}

\makeatother

\providecommand{\claimname}{Claim}
\providecommand{\corollaryname}{Corollary}
\providecommand{\definitionname}{Definition}
\providecommand{\examplename}{Example}
\providecommand{\lemmaname}{Lemma}
\providecommand{\problemname}{Problem}
\providecommand{\propositionname}{Proposition}
\providecommand{\remarkname}{Remark}
\providecommand{\theoremname}{Theorem}

\begin{document}
\title{Independent sets, cliques, and colorings in graphons}
\author{Jan Hladký}
\address{Institute of Mathematics of the Czech Academy of Sciences, Žitná 25,
Praha, Czechia. The Institute of Mathematics is supported by RVO:67985840.
\emph{This work was done while affiliated with} Institut für Geometrie,
TU Dresden, 01062 Dresden, Germany.}
\thanks{\emph{Hladký} was supported by the Alexander von Humboldt Foundation}
\email{honzahladky@gmail.com}
\author{Israel Rocha}
\address{The Czech Academy of Sciences, Institute of Computer Science, Pod
Vodárenskou v\v{e}ží 2, 182~07 Prague, Czech Republic. With institutional
support RVO:67985807.}
\thanks{\emph{Rocha} was supported by the Czech Science Foundation, grant
number GJ16-07822Y}
\email{israelrocha@gmail.com}
\begin{abstract}
We study graphon counterparts of the chromatic and the clique number,
the fractional chromatic number, the $b$-chromatic number, and the
fractional clique number. We establish some basic properties of the
independence set polytope in the graphon setting, and duality properties
between the fractional chromatic number and the fractional clique
number. We present a notion of perfect graphons and characterize them
in terms of induced densities of odd cycles and its complements.
\end{abstract}

\maketitle

\section{Introduction}

The concepts of independent sets, cliques, and colorings are among
the most studied in graph theory. Before stating its graphon counterparts,
let us recall some fundamental concepts.

\subsection{Review of basic properties for graphs}

Suppose that $G=(V,E)$ is a graph. We say that a  function $f:V\rightarrow[k]$
is a \emph{proper coloring} of $G$ with $k$ colors if for all pairs
$xy\in{V \choose 2}$ we have $xy\notin E$ or $f(x)\neq f(y)$. The
\emph{chromatic number} $\chi(G)$ is defined as the minimal number
of colors for which a proper coloring exists. Thus, in a proper coloring
we color the vertices avoiding that neighbors share the same color.

One can generalize the concept of proper coloring by adding multiple
colors to each vertex of a graph in the following manner. Let $b\in\mathbb{N}$.
A map $p:V\rightarrow{[k] \choose b}$ is a \emph{$b$-fold coloring}
of $G$ with $k$ colors if for all pairs $xy\in{V \choose 2}$ we
have $xy\notin E$ or $p(x)\cap p(y)=\emptyset$. The \emph{$b$-fold
chromatic number} $\chi_{b}(G)$ is the smallest number of colors
necessary to construct a proper $b$-fold coloring of $G$.

A set $A\subseteq V$ is an \emph{independent set} of $G$ if for
all pairs $xy\in{A \choose 2}$ we have $xy\notin E$. The size of
the largest independent set of $G$ is the \textit{independence number}
$\alpha(G)$. Dual to the concept of an independent set is that of
a \textit{clique} of a graph $G$, a set $A\subseteq V$ where for
all pairs $xy$ inside $A$ we have $xy\in E$. The \emph{clique number}
$\omega(G)$ is the size of the largest clique of $G$. 

Now denote by $\mathcal{I}(G)$ the collection of independent sets
in $G$. A function $c:\mathcal{I}(G)\rightarrow[0,1]$ is a \emph{fractional
coloring} of $G$ if for every $v\in V(G)$ we have 
\begin{equation}
\sum_{I\in\mathcal{I}(G),v\in I}c(I)\ge1\;.\label{eq:fraccolorfinite}
\end{equation}
The \textit{fractional chromatic number} is then defined as the infimum
of $\sum_{I\in\mathcal{I}(G)}c(I)$ taken over all fractional colorings
$c$. Fractional colorings are indeed a fractional relaxation of ordinary
colorings. Indeed, when $c:\mathcal{I}(G)\rightarrow\{0,1\}$, then~(\ref{eq:fraccolorfinite})
can be interpreted as the condition that in a coloring every vertex
has to be covered by at least one independent set.

We say that a function $f:V\rightarrow[0,+\infty)$ is a \emph{fractional
clique} if for every $I\in\mathcal{I}(G)$ we have
\begin{equation}
\sum_{v\in I}f(v)\leq1.\label{eq:fraccliqueFinte}
\end{equation}
The \emph{fractional clique number} of $G$ is defined as 
\[
\omega_{\mathrm{frac}}(G)=\max\sum_{v\in V}f(v),
\]
where the maximum is taken over all fractional cliques of $G$. In
the same fashion as in fractional colorings, the fractional clique
number is the relaxation of the problem of finding a maximum clique
in a graph. Notice that when $f:V\rightarrow\{0,1\}$, then~(\ref{eq:fraccliqueFinte})
can be interpreted as the condition that every independent set has
at most one vertex of any clique.

\medskip{}

The study of these parameter, and in particular the study of the interplay
of the integral and fractional versions of these parameters is central
in graph theory.

\subsection{Our contribution\label{subsec:OurContribution}}

We translate the basics of the theory regarding independent sets,
cliques, and colorings to the setting of graphons (see Section~\ref{sec:Notation}
for basics). While turning a graph definition into a graphon one is
typically straightforward (and in many cases had been done previously),
counterparts to many natural facts from finite graphs turned out to
be quite challenging. By optimizing over all objects of this type
(independent sets, fractional cliques, \ldots ), we can define the
corresponding numerical graphon parameter (independence number, fractional
clique number, \ldots ). We study relations between these parameters.
We also study continuity of these graphon parameters with respect
to the cut-norm. It turns out that none of the quantities we introduce
is continuous, as is shown in all cases by the sequence of constant
graphons $\left(Y_{n}\equiv\frac{1}{n}\right)_{n}$ converging to
$Y\equiv0$. However, most of these parameters are semicontinuous
with respect to the cut-norm. These results are summarized in Table~\ref{tab:summaryofresults}.

\begin{table}[h]
\begin{tabular}{|c|ccc|}
\hline 
 & (A) representation & (B) semicontinuity & (C) supremum over graphs\tabularnewline
\hline 
independence & \emph{trivial} & Corollary~\ref{cor:indepsemicontinuous} & see Section~\ref{subsec:CooleyKangPikhurko}\tabularnewline
chromatic & Proposition~\ref{prop:chifiniteVSgraphon} & Theorem~\ref{thm:chromaticLimit}(a) & Proposition~\ref{prop:k-partite}\tabularnewline
fractional chromatic & Proposition~\ref{prop:chiFRACfiniteVSgraphon} & \emph{not true} (Example~\ref{ex:leader}) & \emph{not true} (Example~\ref{ex:leader})\tabularnewline
clique & \emph{trivial} & Proposition~\ref{prop:cliquesemicont} & \emph{trivial}\tabularnewline
fractional clique & Proposition~\ref{prop:f-cliqueRepresentation} & Theorem~\ref{thm:liminfFracClique} & Corollary~\ref{cor:fraccliqSUPfinite}\tabularnewline
\hline 
\end{tabular}

\caption{\label{tab:summaryofresults}Summary of results regarding graphon
parameters we study. Column~(A) refers to results that show equality
of the graphon parameter to the corresponding graph parameter in case
of graphon representation of a finite graph. Column~(B) refers to
results that show that these graphon parameters are lower semicontinuous
(for chromatic number, fractional chromatic number, clique number,
and fractional clique number) or upper semicontinuous (for independence
number) in the cut-distance. In all cases the sequence $Y_{n}\equiv\frac{1}{n}\rightarrow Y\equiv0$
shows that we do not have the complementary semicontinuity. Column~(C)
refers to results that show that for a general graphon, the value
can be computed as the supremum over all finite graphs that appear
in that graphon of the graph version of that parameter.}
\end{table}

Further, in Section~\ref{subsec:StructIndep} we introduce a graphon
counterpart to the independence set polytope, in Section~\ref{sec:DualityCliqueColoring}
we treat the LP duality between fractional cliques and fractional
colorings, and in Section~\ref{sec:PerfectGraphons} we introduce
two notions of perfect graphons. Several fairly basic problems remain
open.
\begin{rem}
All our notions only depend on the support of a graphon. That is,
replacing a graphon for example by the indicator function of its support,
the notions of independent sets, colorings, etc., the corresponding
numerical parameters do no change. The only exception to this is a
notion of perfect graphons introduced in Section~\ref{sec:PerfectGraphons}.
\end{rem}

The paper is organized as follows: in Section \ref{sec:Notation}
we provide some notation and prove a preliminary lemma; In Section
\ref{sec:Independent-sets} we study independent sets; Section \ref{sec:Coloring-concepts}
is dedicated to develop different concepts of chromatic number; In
Section \ref{sec:Cliques} we introduce notions of clique numbers;
In Section \ref{sec:Duality}, we prove certain duality properties
between chromatic and clique parameters; and finally, Section \ref{sec:PerfectGraphons}
is devoted to the study of perfect graphons.

\section{Notation and preliminaries\label{sec:Notation}}

Throughout, we fix an atomless Borel probability space $\Omega$ equipped
with a measure $\nu$ (defined on an implicit sigma-algebra). For
$k\in\mathbb{N}$, we denote by $\nu^{\otimes k}$ the product measure
on $\Omega^{k}$.

Graphons, introduced in~\cite{Lovasz2006,Borgs2008c}, are analytic
objects that capture limit properties of dense graphs. We assume the
reader's familiarity with the basics of the theory. Our notation mostly
follows Lovász' treatise~\cite{Lovasz2012}. Our graphons will be
mostly defined on $\Omega^{2}$ In particular, we shall work with
the \emph{density }and the \emph{induced density }of a finite graph
$H$ in a graphon $W$, defined by
\begin{align*}
t(H,W) & =\int_{\left\{ x_{v}\right\} _{v\in V(H)}}\prod_{uv\in E(H)}W\left(x_{u},x_{v}\right)\quad\text{and}\\
t_{\mathrm{ind}}(H,W) & =\int_{\left\{ x_{v}\right\} _{v\in V(H)}}\prod_{uv\in E(H)}W\left(x_{u},x_{v}\right)\cdot\prod_{uv\not\in E(H)}\left(1-W\left(x_{u},x_{v}\right)\right)\;.
\end{align*}
Also, we shall make use of \emph{inhomogeneous random graphs} $\mathbb{G}(n,W)$
described for example in Section~10.1 of~\cite{Lovasz2012}. By
a \emph{subgraphon of $W$ obtained by restricting to a set $A\subset\Omega$}
of positive measure we mean a function $W[A]:A^{2}\rightarrow[0,1]$
which is simply the restriction $W\restriction_{A\times A}$. When
working with this notion, we need to turn $A$ into a probability
space. That is, we view $W[A]$ as a graphon on the probability space
$A$ endowed with measure $\nu_{A}(B):=\frac{\nu(B)}{\nu(A)}$ for
every measurable set $B\subset A$.

All subsets of $\Omega$ or of $\Omega^{2}$ considered will be measurable;
whenever a new set is constructed it follows immediately from the
construction that the set is measurable. For sets $A,B\subset\Omega$,
$C,D\subset\Omega^{2}$ we write $A=B\mod0$, $C=D\mod0$ for equality
up to null sets, i.e., if $\nu(A\triangle B)=0$ and $\nu^{\otimes2}(A\triangle B)=0$.
Given a function $f:X\rightarrow\mathbb{R}$, we denote by $\mathrm{supp}\left(f\right)$
the support of $f$, $\mathrm{supp}\left(f\right):=\{x\in X:f(x)\neq0\}$.

A graphon $W:\Omega\times\Omega\rightarrow[0,1]$ is a \emph{graphon
representation} of a finite graph $G$ if there exists a partition
$\Omega=\bigsqcup_{v\in V(G)}\Omega_{v}$ of sets of measure $\frac{1}{|V(G)|}$
such that $W$ restricted to $\Omega_{u}\times\Omega_{v}$ is either
constant~$0$ or constant~$1$ (modulo a nullset), depending on
whether $uv\notin E(G)$ or $uv\in E(G)$.

\subsection{Structural properties of graphons with a given subgraph}

In this section we state Lemma~\ref{lem:rectangles} which says that
if $t(H,W)>0$ for some finite graph $H$ and some graphon $W$ then
we can find a subgraphon similar to the adjacency matrix of $H$ in
$W$. The case $H=C_{2\ell+1}$ of Lemma~\ref{lem:rectangles} appears
in~\cite{DoHl:Polytons}. Our proof of Lemma~\ref{lem:rectangles}
closely follows~\cite[Lemma 11]{DoHl:Polytons}.
\begin{lem}
\textcolor{red}{\label{lem:rectangles}}Suppose that $W:\Omega\times\Omega\rightarrow[0,1]$
is a graphon. Suppose that $H$ is a graph on vertex set $[k]$ with
the property that $t(H,W)>0$. For each $\epsilon>0$ there exist
$\alpha>0$ and pairwise disjoint sets $A_{1},\ldots,A_{k}\subset\Omega$,
each of measure $\alpha$, such that for each $ij\in E(H)$, the graphon
$W$ is positive everywhere on $A_{i}\times A_{j}$ except a set of
measure at most $\epsilon\alpha^{2}$.
\end{lem}

\begin{proof}
Let $n\in\mathbb{N}$ be such that 
\begin{equation}
n>\frac{k^{2}}{t(H,W)}\;.\label{eq:findn}
\end{equation}
Take an arbitrary partition $\Omega=\bigsqcup_{i=1}^{n}\Omega_{i}$
of $\Omega$ into pairwise disjoint sets of measure $\tfrac{1}{n}$.
Set 
\begin{equation}
D:=\left\{ \mathbf{x}\in\Omega^{k}\::\:\text{there are }i,j\in[k]\text{ and }\ell\in[n]\text{ such that }i\neq j\text{ and }x_{i},x_{j}\in\Omega_{\ell}\right\} \;.
\end{equation}
Then we have 
\begin{equation}
\int_{\mathbf{x}\in D}\prod_{1\le i<j\le k,ij\in E(H)}W(x_{i},x_{j})\le\nu^{\otimes k}\left(D\right)\le\sum\limits _{\substack{i,j=1,\ldots,k\\
i\neq j
}
}\frac{1}{n}\le\frac{k^{2}}{n}\stackrel{\eqref{eq:findn}}{<}t(H,W)\;,\label{eq:IntegralOverDiagonal}
\end{equation}
From~(\ref{eq:IntegralOverDiagonal}), we get 
\[
\int_{\mathbf{x}\in\Omega^{k}\setminus D}\prod_{1\le i<j\le k,ij\in E(H)}W(x_{i},x_{j})>0\;.
\]
Using the definition of $D$, we get that there are pairwise distinct
integers $\ell_{1},\ldots,\ell_{k}\in[n]$ such that 
\[
\int_{\Omega_{\ell_{1}}}\int_{\Omega_{\ell_{2}}}\cdots\int\limits _{\Omega_{\ell_{k}}}\prod_{1\le i<j\le k,ij\in E(H)}W(x_{i},x_{j})>0\;.
\]
We conclude that the set 
\begin{equation}
E:=\left\{ \mathbf{x}\in\Omega_{\ell_{1}}\times\Omega_{\ell_{2}}\times\ldots\times\Omega_{\ell_{k}}\colon\prod_{1\le i<j\le k,ij\in E(H)}W(x_{i},x_{j})>0\right\} \label{eq:indi}
\end{equation}
has positive measure.

Given $\epsilon>0$, let $\delta>0$ be such that 
\begin{equation}
\frac{\nu^{\otimes k}(E)-\delta}{\nu^{\otimes k}(E)+\delta}\ge1-\frac{\epsilon}{2}\;.\label{eq:coted}
\end{equation}
Recall that the $\sigma$-algebra of all measurable subsets of $\Omega_{\ell_{1}}\times\ldots\times\Omega_{\ell_{k}}$
is generated by the algebra consisting of all finite unions of boxes.
Thus there is a finite union $S=\bigcup\limits _{i=1}^{m}R_{i}$ of
boxes $R_{1},\ldots,R_{m}$ in $\Omega_{\ell_{1}}\times\ldots\times\Omega_{\ell_{k}}$
such that $\nu^{\otimes k}(E\setminus S)+\nu^{\otimes k}(S\setminus E)\le\delta$.
Without loss of generality, we may assume that the boxes $R_{1},\ldots,R_{m}$
are pairwise disjoint. Then we have 
\begin{equation}
\frac{\nu^{\otimes k}(E\cap S)}{\nu^{\otimes k}(S)}\ge\frac{\nu^{\otimes k}(E)-\nu^{\otimes k}(E\setminus S)}{\nu^{\otimes k}(E)+\nu^{\otimes k}(S\setminus E)}\ge\frac{\nu^{\otimes k}(E)-\delta}{\nu^{\otimes k}(E)+\delta}\stackrel{\eqref{eq:coted}}{\ge}1-\frac{\epsilon}{2}\;.\label{eq:nemamponeti}
\end{equation}
The left-hand side of~(\ref{eq:nemamponeti}) can be expressed as
\[
\frac{\nu^{\otimes k}(E\cap S)}{\nu^{\otimes k}(S)}=\sum\limits _{i=1}^{m}\frac{\nu^{\otimes k}(R_{i})}{\nu^{\otimes k}(S)}\cdot\frac{\nu^{\otimes k}(E\cap R_{i})}{\nu^{\otimes k}(R_{i})}\;,
\]
i.e.\ as a convex combination of $\frac{\nu^{\otimes k}(E\cap R_{i})}{\nu^{\otimes k}(R_{i})}$,
$i=1,\ldots,m$. Therefore, there is an index $i_{0}\in[m]$ such
that 
\begin{equation}
\frac{\nu^{\otimes k}(E\cap R_{i_{0}})}{\nu^{\otimes k}(R_{i_{0}})}\ge1-\frac{\epsilon}{2}\;.\label{eq:nef}
\end{equation}
Let $R_{i_{0}}$ be of the form $R_{i_{0}}=B_{1}\times\ldots\times B_{k}$.
Find a natural number $p$ such that 
\begin{equation}
p\ge\frac{2k}{\epsilon\nu^{\otimes k}(R_{i_{0}})}\;.\label{eq:train}
\end{equation}
For every $i=1,\ldots,k$, we consider a finite decomposition $B_{i}=B_{i}^{0}\cup\bigcup\limits _{j=1}^{q_{i}}B_{i}^{j}$
of $B_{i}$ into pairwise disjoint sets, such that $\nu(B_{i}^{0})\le\tfrac{1}{p}$
and $\nu(B_{i}^{j})=\tfrac{1}{p}$ for $j=1,\ldots,q_{i}$. Then we
clearly have 
\begin{equation}
\nu^{\otimes k}\left(R_{i_{0}}\setminus\prod\limits _{i=1}^{k}\bigcup\limits _{j=1}^{q_{i}}B_{i}^{j}\right)\le\frac{k}{p}\;,\label{eq:align}
\end{equation}
and so 
\begin{equation}
\frac{\nu^{\otimes k}\left(E\cap\prod\limits _{i=1}^{k}\bigcup\limits _{j=1}^{q_{i}}B_{i}^{j}\right)}{\nu^{\otimes k}\left(\prod\limits _{i=1}^{k}\bigcup\limits _{j=1}^{q_{i}}B_{i}^{j}\right)}\stackrel{\eqref{eq:align}}{\ge}\frac{\nu^{\otimes k}(E\cap R_{i_{0}})-\tfrac{k}{p}}{\nu^{\otimes k}(R_{i_{0}})}\stackrel{\eqref{eq:nef}}{\ge}1-\frac{\epsilon}{2}-\frac{k}{p\nu^{\otimes k}(R_{i_{0}})}\stackrel{\eqref{eq:train}}{\ge}1-\epsilon\;.\label{eq:afteralign}
\end{equation}
The left-hand side of (\ref{eq:afteralign}) can be expressed as the
following convex combination: 
\[
\frac{\nu^{\otimes k}\left(E\cap\prod\limits _{i=1}^{k}\bigcup\limits _{j=1}^{q_{i}}B_{i}^{j}\right)}{\nu^{\otimes k}\left(\prod\limits _{i=1}^{k}\bigcup\limits _{j=1}^{q_{i}}B_{i}^{j}\right)}=\sum\limits _{j_{1}=1}^{q_{1}}\cdots\sum\limits _{j_{k}=1}^{q_{k}}\frac{\nu^{\otimes k}\left(\prod\limits _{i=1}^{k}B_{i}^{j_{i}}\right)}{\nu^{\otimes k}\left(\prod\limits _{i=1}^{k}\bigcup\limits _{j=1}^{q_{i}}B_{i}^{j}\right)}\cdot\frac{\nu^{\otimes k}\left(E\cap\prod\limits _{i=1}^{k}B_{i}^{j_{i}}\right)}{\nu^{\otimes k}\left(\prod\limits _{i=1}^{k}B_{i}^{j_{i}}\right)}\;.
\]
Therefore by (\ref{eq:afteralign}), there are indices $j_{i}\in[q_{i}]$,
$i=1,\ldots,k$, such that
\begin{equation}
\frac{\nu^{\otimes k}\left(\prod\limits _{i=1}^{k}B_{i}^{j_{i}}\setminus E\right)}{\nu^{\otimes k}\left(\prod\limits _{i=1}^{k}B_{i}^{j_{i}}\right)}\le\epsilon\;.\label{eq:zmenajezivot}
\end{equation}
We set $A_{i}=B_{i}^{j_{i}}$ for $i=1,\ldots,k$. Then $A_{1},\ldots,A_{k}$
are pairwise disjoint (as $A_{i}\subseteq\Omega_{\ell_{i}}$ for every
$i$), and each of these sets has the same measure $\alpha=\tfrac{1}{p}$.

By~(\ref{eq:indi}) we have 
\begin{equation}
\left\{ (x_{1},x_{2},\ldots,x_{k})\in A_{1}\times A_{2}\times\ldots\times A_{k}\colon\prod_{1\le i<j\le k,ij\in E(H)}W(x_{i},x_{j})=0\right\} =\prod\limits _{i=1}^{k}B_{i}^{j_{i}}\setminus E\;.\label{eq:poui}
\end{equation}
By~(\ref{eq:zmenajezivot}), 
\[
\nu^{\otimes k}\left(\prod\limits _{i=1}^{k}B_{i}^{j_{i}}\setminus E\right)\le\epsilon\nu^{\otimes k}\left(\prod\limits _{i=1}^{k}B_{i}^{j_{i}}\right)=\frac{\epsilon}{p^{k}}\;.
\]
Consider an arbitrary edge $ij\in E(H)$. Observe that the set $\prod\limits _{i=1}^{k}B_{i}^{j_{i}}\setminus E$
contains all $k$-tuples $\mathbf{x}\in\prod_{h=1}^{k}A_{h}$ that
have the property that $W(x_{i},x_{j})=0$. Therefore, we get from~(\ref{eq:poui})
that 
\[
\nu^{\otimes2}\left(\left\{ (x_{i},x_{j})\in A_{i}\times A_{j}\colon W(x_{i},x_{j})=0\right\} \right)\le\epsilon\alpha^{2}\;,
\]
as required. 
\end{proof}

\section{\label{sec:Independent-sets}Independent sets}

While classically, the notion of independent sets and cliques are
in one-to-one correspondence by taking complements, our definitions
for graphons look at each of these concepts at a different scale.
The independence number of a graphon is a number in~$[0,1]$ that
should be interpreted as the fraction of vertices in a maximum independent
set of a graph that corresponds to that graphon. That is, we scale
down the independence number linearly. On the other hand, in Section~\ref{sec:Cliques}
we define the clique number (and its fractional variant) which are
all absolute with no additional rescaling introduced (see Section~\ref{sec:Cliques}
for a discussion of subtleties). If a graphon contains a copy of $K_{17}$
then its clique number will be at least~$17$. Of course, which of
the two parameters is scaled and which one is not is just a matter
of convention. Let us remark that in~\cite{DHM:CliquesRandom} a
different scale of $\log n$ is put on these parameters and studied
in the context of inhomogeneous random graphs $\mathbb{G}(n,W)$.

The following is then the obvious graphon counterpart to independent
sets.
\begin{defn}
Let $W:\Omega\times\Omega\rightarrow[0,1]$ be a graphon. A set $A\subseteq\Omega$
is an \emph{independent set} if $W$ is zero almost everywhere on
$A\times A$. Denote by $\mathcal{I}(W)$ the set of independent sets
of a graphon $W$ and for each $x\in\Omega$ denote by $\mathcal{I}_{x}(W)\subseteq\mathcal{I}(W)$
the set of independent sets of $W$ containing the point $x$. 
\end{defn}

The definition of $\mathcal{I}_{x}(W)$ may look suspicious as it
involves a measure-zero condition (a point belonging to a set). This
will not cause a problem since we shall always work with collections
$\mathcal{I}_{x}(W)$ for a set of $x$'s of positive measure. 

Let us mention that the paper \cite{Bachoc} is closely related to
ours, though concerned with different questions. In that paper, the
authors investigate the independence ratio (which we call independence
number in our setting) and chromatic number for bounded self-adjoint
operators on an $L^{2}$ space. In particular, the graphon operator
as defined in \cite[Section 7.5]{Lovasz2012} is an example of such
operators. In particular, our notion of independent set of a graphon
is equivalent to the notion of independent set as introduced in~\cite{Bachoc}
of the graphon operator associated to that graphon. In~\cite{Bachoc},
the authors provide bounds for these parameters in terms of the eigenvalues
of the operator \textemdash{} an analog to the Hoffman bound \textemdash{}
and use harmonic analysis and convex optimization to study packing
and coloring problems for finite and infinite graphs. 

Here, we need to describe particular properties of independent sets
in order to understand the behavior of the chromatic number and the
clique number for graphons.

The next lemma which appears in~\cite[Lemma 20]{HlHuPi:Komlos} asserts
that the weak{*} limit of independent sets is again an independent
set.
\begin{lem}
\label{lem:lemmaIndependentSets}Let $W$ be a graphon. Suppose $\left(A_{n}\right)_{n}$
is a sequence of sets in $\Omega$ with the property that 
\[
\lim_{n\rightarrow\infty}\int_{A_{n}\times A_{n}}W=0.
\]
Suppose that the indicator functions of the sets $A_{n}$ converge
weak{*} to a function $f$. Then $\mathrm{supp}\left(f\right)$ is
an independent set in $W$. 
\end{lem}

It follows that for a convergent sequence of graphons, the weak{*}
limit of independent sets in the sequence form an independent set
in the limit graphon.
\begin{cor}
\label{cor:independentSetsConvergence}Let $W_{n}:\Omega\times\Omega\rightarrow[0,1]$
be a sequence of graphons converging in the cut-norm to $W:\Omega\times\Omega\rightarrow[0,1]$.
Let $I_{n}\subset\Omega$ be an independent set in $W_{n}$. Suppose
that the indicator functions of the sets $I_{n}$ converge weak{*}
to a function $f$. Then $\mathrm{supp}\left(f\right)$ is an independent
set in $W$. 
\end{cor}

\begin{proof}
Notice that 
\[
\int_{I_{n}\times I_{n}}W=\left|\int_{I_{n}\times I_{n}}\left(W_{n}-W\right)\right|\leq\left\Vert W_{n}-W\right\Vert _{\Square}.
\]
Thus, $\lim_{n\rightarrow\infty}\int_{I_{n}\times I_{n}}W=0$ and
the claim follows from Lemma~\ref{lem:lemmaIndependentSets}.
\end{proof}
The defining property of weak{*} convergence gives us that in the
setting of the corollary above, we have 
\begin{equation}
\lim_{n}\nu(I_{n})=\lim_{n}\int_{\Omega}\mathbf{1}_{I_{n}}=\int_{\Omega}f\le\nu\left(\mathrm{supp}\left(f\right)\right)\;,\label{eq:uio}
\end{equation}
where the last inequality uses that $f$ is bounded above by~1, since
it is a weak{*} limit of functions bounded above by~1. Thus, as a
consequence of Corollary~\ref{cor:independentSetsConvergence} we
get that the supremum of the measures of independent sets in a graphon
is attained. This leads us to the following definition.
\begin{defn}
The measure of the largest independent set of a graphon $W:\Omega\times\Omega\rightarrow[0,1]$
is called the \emph{independence number} of $W$ and denoted by $\alpha(W)$.
\end{defn}

Corollary~\ref{cor:independentSetsConvergence} yields upper semicontinuity
of the independence number. As mentioned in Section~\ref{subsec:OurContribution},
the sequence $Y_{n}\equiv\frac{1}{n}\rightarrow Y\equiv0$ shows that
we do not have lower semicontinuity in general.
\begin{cor}
\label{cor:indepsemicontinuous}Suppose that $\left(W_{n}\right)_{n}$
is a sequence of graphons that converges to $W$ in the cut-distance.
Then $\alpha(W)\ge\limsup_{n}\alpha\left(W_{n}\right)$.
\end{cor}

\begin{proof}
We may as well assume that $\left(W_{n}\right)_{n}$ converges to
$W$ in the cut-norm, and that the limit $\lim_{n}\alpha\left(W_{n}\right)$
exists. Now, for each $n$, consider an independent set $I_{n}$ in
$W_{n}$ of size $\alpha\left(W_{n}\right)$. Let $f$ be a weak{*}
accumulation point of the sequence of indicator functions of the sets
$\left(I_{n}\right)_{n}$; at least one such accumulation point exists
by the sequential Banach\textendash Alaoglu Theorem. By Corollary~\ref{cor:independentSetsConvergence},
$\mathrm{supp}\left(f\right)$ is an independent set, and by the same
calculation as in~(\ref{eq:uio}), the measure of $\mathrm{supp}\left(f\right)$
is at least $\lim_{n}\alpha\left(W_{n}\right)$.
\end{proof}
\medskip{}

\subsection{Structure of independent sets\label{subsec:StructIndep}}

In this section, we make some observations about the structure of
independent sets in a graphon. Often, rather than dealing with all
independent sets, it is convenient to restrict attention just to maximal
ones, from which all the remaining ones can be easily recovered. We
denote the set of maximal independent sets (modulo nullsets) in a
graphon $W$ by $\mathcal{I}_{\mathrm{max}}(W)\subset\mathcal{I}(W)$.
Even the set $\mathcal{I}_{\mathrm{max}}(W)$ can be quite complicated,
at least with respect to its cardinality. Indeed, let $W$ be a graphon
representing a disjoint union of countably many complete bipartite
graphs $\left(H_{n}=A_{n}\sqcup B_{n}\right)_{n=1}^{\infty}$, which
occupy measure $2^{-n}$ each. Then the maximal independent sets in
$W$ are all unions of maximal independent sets in all graphs $H_{n}$,
(for example $A_{1},A_{2},A_{3},B_{4},B_{5},\ldots$), of which there
are uncountably many.

Another perspective on the structure of $\mathcal{I}(W)$ comes from
polyhedral combinatorics. To motivate this, let us first recall an
approach common for finite graphs. Given a finite graph $G$, each
independent set $I$ in $G$ can viewed as a vector in $\{0,1\}^{V(G)}\subset\mathbb{R}^{V(G)}$.
Taking the convex hull of all such vectors, one gets what is called
an \emph{independent set }(or \emph{stable set})\emph{ polytope} $\mathrm{IND}(G)\subset\mathbb{R}^{V(G)}$.
Among many classical results about $\mathrm{IND}(G)$, let us mention
perhaps the most basic one.
\begin{prop}[{e.g. \cite[(9.1.3)]{MR936633}}]
\label{prop:finiteIND}For any graph $G$, each point $\mathbf{x}\in\mathrm{IND}(G)$
satisfies the following two families of conditions:
\begin{align}
0\le\mathbf{x}_{v} & \le1\quad\text{for each \ensuremath{v\in V(G)} and}\label{eq:IND1}\\
\mathbf{x}_{u}+\mathbf{x}_{v} & \le1\quad\text{for each \ensuremath{uv\in E(G)}\;.}\label{eq:IND2}
\end{align}
Further, (\ref{eq:IND1}) and (\ref{eq:IND2}) characterize $\mathrm{IND}(G)$
if and only if $G$ is bipartite.
\end{prop}

Now, in the graphon setting we proceed as follows. We represent each
set $I\in\mathcal{I}(W)$ of a graphon $W:\Omega\times\Omega\rightarrow[0,1]$
by its characteristic function, which we view as an element in $\{0,1\}^{\Omega}\subset\mathbb{R}^{\Omega}$.
We can now take the closure (in the weak{*} topology) of the convex
hull of such functions and get what we call \emph{independent set
polyton} $\mathrm{IND}(W)\subset\mathbb{R}^{\Omega}$. Such a graphon
approach to polyhedral combinatorics has been introduced~\cite{DoHl:Polytons},
namely for the so-called matching polytope/polyton. To illustrate
the potential of this area, let us prove a part of a counterpart to
Proposition~\ref{prop:finiteIND}.
\begin{prop}
\label{prop:graphonIND}For any graph $W:\Omega\times\Omega\rightarrow[0,1]$,
each point $\mathbf{x}\in\mathrm{IND}(W)$ satisfies the following
two families of conditions:
\begin{align}
0\le\mathbf{x}_{v} & \le1\quad\text{for almost all \ensuremath{v\in\Omega} and}\label{eq:IND1W}\\
\mathbf{x}_{u}+\mathbf{x}_{v} & \le1\quad\text{for almost all \ensuremath{uv\in\mathrm{supp}\left(W\right)}\;.}\label{eq:IND2W}
\end{align}
Further, if $W$ is not bipartite, then (\ref{eq:IND1W}) and (\ref{eq:IND2W})
do not characterize $\mathrm{IND}(W)$.
\end{prop}

\begin{problem}
In analogy with Proposition~\ref{prop:finiteIND}, establish the
other direction in Proposition~\ref{prop:graphonIND}.
\end{problem}

\begin{proof}[Proof of Proposition~\ref{prop:graphonIND}]
Obviously, every indicator function of an independent set $I\in\mathcal{I}(W)$
satisfies~(\ref{eq:IND1W}) and~(\ref{eq:IND2W}). These inequalities
are then inherited to convex combinations and closure, thus proving
the first part. 

Suppose now that $W$ is not bipartite. We shall prove that the point
$\mathbf{y}\equiv\frac{1}{2}$ is not in $\mathrm{IND}(W)$, even
though it apparently satisfies~(\ref{eq:IND1W}) and~(\ref{eq:IND2W}).
By \cite[Proposition 21]{DoHl:Polytons} (stated also as Proposition~\ref{prop:k-partite}
below), we know that $t(C_{2\ell+1},W)>0$ for some $\ell\in\mathbb{N}$.
By Lemma~\ref{lem:rectangles} we know that there exist disjoint
sets $A_{1},A_{2},\ldots,A_{2\ell+1}\subset\Omega$ of the same measure
$\alpha>0$ such that that $W$ is positive everywhere on $A_{i}\times A_{j}$
except a set of measure at most $\frac{\alpha^{2}}{10\ell}$. Now,
suppose for contradiction that $\mathbf{y}$ is in the weak{*} closure
of convex combinations of characteristic functions of independent
sets. In particular, there exists a point $\mathbf{y}^{*}$ such that
we have 
\begin{equation}
\int_{\omega\in\bigcup_{i=1}^{2\ell+1}A_{i}}\mathbf{y}^{*}(\omega)>\int_{\omega\in\bigcup_{i=1}^{2\ell+1}A_{i}}\mathbf{y}(\omega)-\frac{\alpha}{10}=\frac{1}{2}(2\ell+1)\alpha-\frac{\alpha}{10}\label{eq:integrallarge}
\end{equation}
 and such that $\mathbf{y}^{*}$ is a convex combination of characteristic
functions of independent sets, $\mathbf{y}^{*}=\sum_{j=1}^{t}\alpha_{j}\mathbf{1}_{I_{j}}$,
where $\alpha_{j}\ge0$ are the convex coefficients and $I_{j}\in\mathcal{I}(W)$.
We get~(\ref{eq:integrallarge}) must hold also for one of the terms
appearing in the convex combination, i.e., there exists $j\in[t]$
such that
\begin{equation}
\nu\left(I_{j}\cap\bigcup_{i=1}^{2\ell+1}A_{i}\right)\ge\frac{1}{2}(2\ell+1)\alpha-\frac{\alpha}{10}\;.\label{eq:larmeas}
\end{equation}
Call an index $i\in[2\ell+1]$ \emph{marked }if $\nu\left(I_{j}\cap A_{i}\right)\ge\frac{\alpha}{5\ell}$.
Observe that we cannot have two consecutive marked indices $i$ and
$i+1$ (with labeling considered is modulo $2\ell+1$), since $W_{\restriction A_{i}\times A_{i+1}}$
is positive on most of the domain and $I_{j}$ is an independent set.
Hence, there are at most $2\ell$ marked indices. For a marked index
$i$ we have $\nu(I_{j}\cap A_{i})\le\alpha$ and for an unmarked
index $i$ we have $\nu(I_{j}\cap A_{i})\le\frac{\alpha}{5\ell}$.
Hence,
\[
\nu\left(I_{j}\cap\bigcup_{i=1}^{2\ell+1}A_{i}\right)\le\ell\cdot\alpha+(\ell+1)\cdot\frac{\alpha}{5\ell}\;.
\]
This contradicts~(\ref{eq:larmeas}).
\end{proof}

\subsection{\label{subsec:CooleyKangPikhurko}A recent result of Cooley, Kang
and Pikhurko}

We originally thought that there is no way how to relate subgraph
densities of a graphon to its independence number. However, recently
Cooley, Kang and Pikhurko~\cite{CKP:SizeConvergence} found a natural
way to fill this gap, which we state below.
\begin{thm}
Let $W:\Omega\times\Omega\rightarrow[0,1]$ be a graphon. For each
$k\in\mathbb{N}$, let $a_{k}$ be the probability that $\mathbb{G}(k,W)$
does not contain any edges, that is $a_{k}=t_{\mathrm{ind}}(I_{k},W)$,
where $I_{k}$ the edgeless graph on $k$ vertices. Then the limit
$\lim_{k\to\infty}\sqrt[k]{a_{k}}$ exists, and is equal to the independence
number of $W$.
\end{thm}

\section{\label{sec:Coloring-concepts}Coloring concepts}

\subsection{Ordinary colorings}

First, we define a counterpart to the usual concept of coloring.

\begin{defn}
\label{def:chi}Let $W:\Omega\times\Omega\rightarrow[0,1]$ be a graphon.
We say that a measurable function $f:\Omega\rightarrow[k]$ is a \emph{proper
coloring} of $W$ with $k$ colors if, for each $i\in[k]$, the set
$f^{-1}(i)$ is an independent set in $W$. The \emph{chromatic number}
$\chi(W)$ is defined as the minimal number of colors for which a
proper coloring exists. Here, as will be with other versions of the
chromatic number, when no proper coloring of $W$ exists, $\chi(W)$
is taken to be infinity. In particular, this notion splits the space
of graphons into graphons of \emph{finite} and \emph{infinite chromatic
number}.
\end{defn}

The next easy proposition shows that Definition~\ref{def:chi} is
consistent with the graph definition of the chromatic number.
\begin{prop}
\label{prop:chifiniteVSgraphon}Suppose $G$ is a finite graph and
let $W:\Omega\times\Omega\rightarrow[0,1]$ be its graphon representation.
Then $\chi(G)=\chi(W)$.
\end{prop}

\begin{proof}
Let $\Omega=\Omega_{1}\sqcup\dots\sqcup\Omega_{n}$ be the partition
for $W$ corresponding to the vertices of $G$. If $U_{1}\sqcup U_{2}\sqcup\ldots\sqcup U_{k}=V(G)$
is a proper $k$-coloring of $G$, then the independent sets $\left(\bigcup_{v\in U_{i}}\Omega_{v}\right)_{i\in[k]}$
give a proper $k$-coloring of $W$. Thus, $\chi(W)\le\chi(G)$. Conversely,
suppose that $I_{1}\sqcup I_{2}\sqcup\ldots\sqcup I_{k}=\Omega$.
Now, for every $j\in[n]$, there exists at least one index, say $\ell_{j}\in[k]$,
such that $\Omega_{j}\cap I_{\ell_{j}}$ has positive measure. Obviously,
the map $j\mapsto\ell_{j}$ is a proper coloring of $G$.
\end{proof}
The graphons $Y_{n}\rightarrow Y$ from Section~\ref{subsec:OurContribution}
show that the chromatic number is not upper continuous. Indeed, we
have $\chi(Y_{n})=\infty$ and $\chi(Y)=1$. On the other hand, the
following theorem shows lower semicontinuity.

\begin{thm}
\label{thm:chromaticLimit}Let $W:\Omega\times\Omega\rightarrow[0,1]$
be a graphon.
\end{thm}

\begin{enumerate}
\item[(a)] For every sequence of graphs $\left(H_{n}\right)$ that converges
to $W$ in the cut-distance we have $\chi(W)\le\liminf_{n\rightarrow\infty}\chi(H_{n})$.
\item[(b)] For the sequence $\left(G_{n}\sim\mathbb{G}(n,W)\right)$ we have
$\chi(W)=\lim_{n\rightarrow\infty}\chi(G_{n})$, almost surely.
\end{enumerate}
The next proposition, which can be found in~\cite[Proposition 5]{DoHl:Polytons},
is the key towards proving Theorem~\ref{thm:chromaticLimit}. We
remark that the proof of this proposition given in~\cite{DoHl:Polytons}
is short but non-trivial.
\begin{prop}
\label{prop:k-partite}Suppose $k\in\mathbb{N}$ and $W:\Omega\times\Omega\rightarrow[0,1]$
is a graphon. Then $\chi(W)\le k$ if and only if $t(G,W)=0$ for
each finite graph $G$ with $\chi(G)\geq k+1$. 
\end{prop}

\begin{proof}[Proof of Theorem~\ref{thm:chromaticLimit}(a)]
Suppose that $\chi(W)\ge k$. By (the more difficult part of) Proposition~\ref{prop:k-partite},
there exists a graph $F$ with chromatic number at least $k$ for
which $t(F,W)$ is positive, say $a:=t(F,W)$. Since convergence in
the cut-distance implies convergence of subgraph densities, we conclude
that each graph $H_{n}$ has subgraph density of $F$ equal to $a+o_{n}(1)$.
In particular, for $n$ sufficiently large, this density is positive.
We conclude that the chromatic number of such a graph $H_{n}$ is
at least as that of $F$.
\end{proof}
\begin{proof}[Proof of Theorem~\ref{thm:chromaticLimit}(b)]
Since the graphs $G_{n}$ converge to $W$ in the cut-distance almost
surely, we have $\chi(W)\le\liminf_{n\rightarrow\infty}\chi(G_{n})$
by~(a). On the other hand, if $\chi(W)\le k$, then by (the easy
part of) Proposition~\ref{prop:k-partite}, each graph of chromatic
number at least $k+1$ has zero density. That is, the probability
of $W$-sampling any such graph is zero.
\end{proof}

\subsection{Fractional colorings}

The next definition is a straightforward counterpart to the graph
definition of fractional colorings.
\begin{defn}
Suppose that $W:\Omega\times\Omega\rightarrow[0,1]$ is a graphon.
We say that a function $c:\mathcal{I}(W)\rightarrow[0,1]$ is a \emph{fractional
coloring} of $W$ if for almost every $x\in\Omega$ we have $\sum_{I\in\mathcal{I}_{x}(W)}c(I)\ge1$.
The \emph{fractional chromatic number} of $W$ is defined as

\begin{equation}
\chi_{\mathrm{frac}}(W)=\inf\sum_{I\in\mathcal{I}(W)}c(I),\label{eq:deffraccolor}
\end{equation}
where the infimum is taken over the set of all fractional colorings
of $W$.
\end{defn}

Notice that whenever $\sum_{I\in\mathcal{I}(W)}c(I)$ is finite, then
the support of the function $c$ is at most countable. Thus it suffices
to take the infimum over the set of all fractional colorings of $W$
that are not zero at most in a countable subset of $\mathcal{I}(W)$.

The next proposition shows that the definitions of fractional chromatic
number for graphs and graphons are consistent.
\begin{prop}
\label{prop:chiFRACfiniteVSgraphon}Suppose $G$ is a finite graph
and let $W$ be its graphon representation. Then $\chi_{\mathrm{frac}}(G)=\chi_{\mathrm{frac}}(W)$.
\end{prop}

\begin{proof}
Denote by $\mathcal{I}(W)$ and $\mathcal{I}(G)$ the set of independent
sets of $W$ and $G$, respectively. There is a map $\pi:\mathcal{I}(G)\rightarrow\mathcal{I}(W)$
which maps independent sets in $\mathcal{I}(G)$ to corresponding
independent sets in $\mathcal{I}(W)$. If $c_{G}$ is an arbitrary
fractional coloring of $G$ then we can define $c_{W}:\pi\left(\mathcal{I}(G)\right)\rightarrow[0,1]$,
$c_{W}\left(\pi(I)\right):=c_{G}(I)$. Extending $c_{W}$ by zeros
to $\mathcal{I}(W)$, we get a fractional coloring of $W$ of the
same size. This implies $\chi_{\mathrm{frac}}(G)\geq\chi_{\mathrm{frac}}(W)$. 

Now consider a fractional coloring $c_{W}:\mathcal{I}(W)\rightarrow[0,1]$
of $W$. Let $\Omega=\Omega_{1}\sqcup\dots\sqcup\Omega_{n}$ be the
partition for $W$ corresponding to the vertices of $G$. For each
$v\in V(G)$, choose a point $x_{v}\in\Omega_{v}$. We shall assume
that $\sum_{I\in\mathcal{I}_{x_{v}}(W)}c_{W}(I)\geq1$ for each $v\in V(G)$;
this is true for almost all choices of $x_{v}\in\Omega_{v}$. Now
define $c_{G}:\mathcal{I}(G)\rightarrow[0,1]$ by
\[
c_{G}(J):=\sum_{I}c_{W}(I)\;,
\]
where $I$ ranges over all sets in $\bigcap_{v\in J}\mathcal{I}_{x_{v}}(W)\setminus\bigcup_{v\in V(G)\setminus J}\mathcal{I}_{x_{v}}(W)$.
The inclusion-exclusion formula implies that for each $v\in V(G)$
we have $\sum_{J\in\mathcal{I}_{v}(G)}c_{G}(J)=\sum_{I\in\mathcal{I}_{x_{v}}(W)}c_{W}(I)$.
Since $\sum_{I\in\mathcal{I}_{x_{v}}(W)}c_{W}(I)\geq1$, for each
$v\in V(G)$ we have$\sum_{J\in\mathcal{I}_{v}(G)}c_{G}(J)\geq1$.
Therefore, $c_{G}$ is a valid fractional coloring of $G$. Thus $\chi_{\mathrm{frac}}(G)\leq\chi_{\mathrm{frac}}(W)$,
which finishes the proof.
\end{proof}
\begin{rem}
\label{rem:fractcliquesizedoesnotmatter}Observe that in the proof
of Proposition~\ref{prop:chiFRACfiniteVSgraphon} we did not use
that each vertex of $G$ is represented by a set of $\Omega$ of measure
$\frac{1}{|V(G)|}$. That is, the same result holds true if vertices
of $G$ are represented by arbitrary sets $\Omega_{1}\sqcup\dots\sqcup\Omega_{n}=\Omega$
of positive but not necessarily the same measure.
\end{rem}

\begin{example}
\label{ex:leader}In~\cite{MR1358531} Leader constructs a graph
$R$ with a countable vertex set, say $\mathbb{N}$, whose fractional
chromatic number is strictly greater than the supremum of the fractional
chromatic numbers of its finite subgraphs. Considering an arbitrary
partition $\Omega=\bigsqcup_{i\in\mathbb{N}}\Omega_{i}$ into sets
of positive measure, and taking $W$ to be a graphon representation
of $R$ with respct to the partition $\Omega=\bigsqcup_{i\in\mathbb{N}}\Omega_{i}$
, we get by a version\footnote{For countable graphs; the proof of such a version is \emph{mutatis
mutandis}.} of Proposition~\ref{prop:chiFRACfiniteVSgraphon} (see also Remark~\ref{rem:fractcliquesizedoesnotmatter})
that $\chi_{\mathrm{frac}}(W)>\sup\left\{ \chi_{\mathrm{frac}}(G):t(G,W)>0\right\} $.
Note that this example also shows that $\chi_{\mathrm{frac}}(\cdot)$
is not lower semicontinuous. Indeed, for $\ell=1,2,3,\ldots$, we
consider graphon $W_{\ell}$ which is equal to $W$ on $\left(\cup_{i=1}^{\ell}\Omega_{i}\right)\times\left(\cup_{i=1}^{\ell}\Omega_{i}\right)$
and~0 otherwise. Obviously, the graphons $\left(W_{n}\right)_{n}$
converge to $W$ in the cut-norm (actually, even pointwise), but we
have $\lim_{n}\chi_{\mathrm{frac}}(W_{n})=\sup\left\{ \chi_{\mathrm{frac}}(G):t(G,W)>0\right\} <\chi_{\mathrm{frac}}(W)$.
\end{example}

We also note that in~\cite{MR1358531} Leader constructs another
countable graph $T$ in which $\chi_{\mathrm{frac}}(T)$ is not attained
by any fractional coloring. Taking a graphon representation of $T$,
we see that the fractional chromatic number of a graphon need not
be attained.

\subsection{$b$-fold colorings}

Next, we define counterparts of $b$-fold colorings for graphons.
\begin{defn}
Suppose that $b\in\mathbb{N}$. For a graphon $W:\Omega\times\Omega\rightarrow[0,1]$,
a measurable map $p:\Omega\rightarrow{[k] \choose b}$ is a \emph{$b$-fold
coloring} of $W$ with $k$ colors if for almost all pairs $(x,y)\in\Omega\times\Omega$
we have $W(x,y)=0$ or $p(x)\cap p(y)=\emptyset$. The \emph{$b$-fold
chromatic number} $\chi_{b}(W)$ is the smallest number of colors
necessary to construct a proper $b$-fold coloring of $W$.
\end{defn}

As in the graph case, the relation between the fractional chromatic
number, the $1$-fold chromatic number and the ordinary chromatic
number is as follows: $\chi_{\mathrm{frac}}(W)\leq\chi(W)=\chi_{1}(W)$.
The next theorem shows that at least qualitatively, we can reverse
the inequality.
\begin{prop}
Let $W:\Omega\times\Omega\rightarrow[0,1]$ be a graphon. If $\chi_{\mathrm{frac}}(W)<\infty$
then $\chi(W)<\infty$.
\end{prop}

\begin{proof}
Suppose that $c$ is a fractional coloring of $W$ with $\sum_{I\in\mathcal{I}(W)}c(I)<\infty$.
Thus, there is a finite number of independent sets $\left\{ I_{1},\dots,I_{k}\right\} $
such that
\begin{equation}
\sum_{i=1}^{k}c(I_{i})>\sum_{I\in\mathcal{I}(W)}c(I)-1.\label{eq:contradict1}
\end{equation}
 In particular, for each $x\in\Omega$ we have
\[
\sum_{I\in\mathcal{I}_{x}(W)}c(I)<\sum_{i\in[k],x\in I_{i}}c(I_{i})+1\;.
\]
This implies that if $x\in\Omega\setminus\bigcup_{i=1}^{k}I_{i}$
then $\sum_{I\in\mathcal{I}_{x}(W)}c(I)<0+1$. We conclude that $\Omega\setminus\bigcup_{i=1}^{k}I_{i}$
is a nullset. Therefore, $\left\{ I_{1},\dots,I_{k}\right\} $ corresponds
to a proper coloring of $W$with finitely many colors.
\end{proof}
A sequence $\left\{ a_{n}\right\} _{n=1}^{\infty}$ is said to be
\emph{subadditive} if for every $m$ and $n$ we have $a_{m+n}\leq a_{m}+a_{n}$.
Fekete's Lemma is a useful result concerning subadditive sequences.
\begin{lem}[Fekete's Lemma]
For every subadditive sequence ${\displaystyle \left\{ a_{n}\right\} _{n=1}^{\infty}}$,
the limit ${\displaystyle {\displaystyle \lim_{n\to\infty}\frac{a_{n}}{n}}}$
exists and is equal to $\inf\frac{a_{n}}{n}$.
\end{lem}

If we have an $a$-fold coloring $p:\Omega\rightarrow{\{1,\ldots,k\} \choose a}$
of a graphon $W$ which uses the palette $\{1,\ldots,k\}$ and a $b$-fold
coloring $q:\Omega\rightarrow{\{k+1,\ldots,k+\ell\} \choose b}$ which
uses the palette $\{k+1,\ldots,k+\ell\}$, we observe that the map
$x\mapsto p(x)\cup q(x)$ is an $(a+b)$-fold coloring of $W$ which
uses the palette $\{1,\ldots,k+\ell\}$. Therefore, we have $\chi_{a+b}(W)\leq\chi_{a}(W)+\chi_{b}(W)$.
Fekete's lemma tells us that 
\[
\lim_{b\rightarrow\infty}\frac{\chi_{b}(W)}{b}=\inf\frac{\chi_{b}(W)}{b}.
\]
The next theorem tells us that this limit is equal to the fractional
chromatic number of $W$.
\begin{thm}
\label{thm:limb-fold}For every graphon $W$ we have
\[
\chi_{\mathrm{frac}}(W)=\lim_{b\rightarrow\infty}\frac{\chi_{b}(W)}{b}=\inf_{b\in\mathbb{N}}\frac{\chi_{b}(W)}{b}.
\]
\end{thm}

\begin{proof}
First, we prove that $\chi_{\mathrm{frac}}(W)\leq\inf_{b\rightarrow\infty}\frac{\chi_{b}(W)}{b}$.
Given $b\in\mathbb{N}$, fix a proper $b$-fold coloring $c_{b}$
of $\Omega$ for the $b$-chromatic number $\chi_{b}$, say using
$\ell$ colors. We will construct a fractional coloring $c$ such
that 
\[
\sum_{I\in\mathcal{I}(W)}c(I)=\frac{\chi_{b}(W)}{b}.
\]
To this end, for each $j\in[\ell]$, consider the independent set
$I_{j}:=\left\{ x\in\Omega:j\in c_{b}(x)\right\} $, and for each
such set, define $c(I_{j})=1/b$. We have, 
\[
\sum_{I\in\mathcal{I}_{x}(W)}c(I)=1,
\]
for almost all $x\in\Omega$, and 
\[
\sum_{I\in\mathcal{I}(W)}c(I)=\frac{\chi_{b}}{b},
\]
as required. 

It remains to prove that $\chi_{\mathrm{frac}}(W)\ge\inf_{b\rightarrow\infty}\frac{\chi_{b}(W)}{b}$.
We can assume $\chi_{\mathrm{frac}}(W)$ is bounded. The next claim
allows us to restrict ourselves to fractional colorings with weights
that are a multiple of $\frac{1}{b}$.
\begin{claim*}
Suppose that $\epsilon>0$ is fixed. For any $b$ sufficiently large,
there exists a fractional coloring $c:\mathcal{I}(W)\rightarrow\{\frac{j}{b}:j=0,1,\ldots,b\}$
with finite support such that $\sum_{I\in\mathcal{I}(W)}c(I)\le\chi_{\mathrm{frac}}(W)+\epsilon$.
\end{claim*}
\begin{proof}[Proof of Claim]
Take $\delta>0$ such that 
\[
\frac{\chi_{\mathrm{frac}}(W)+\delta}{1-\delta}\leq\chi_{\mathrm{frac}}(W)+\frac{\epsilon}{2}.
\]
Take a fractional coloring $c_{0}$ of $W$ such that $\sum_{I\in\mathcal{I}(W)}c_{0}(I)\leq\chi_{\mathrm{frac}}(W)+\delta.$
Since $\sum_{I\in\mathcal{I}(W)}c_{0}(I)$ is finite, we can find
a finite subset $\mathcal{I}_{0}\mathcal{\subseteq I}(W)$ such that
\begin{equation}
\sum_{I\in\mathcal{I}(W)\setminus\mathcal{I}_{0}}c_{0}(I)<\delta\;.\label{eq:needtomultiply}
\end{equation}
Define a function $c_{1}$ by setting it equal to $c_{0}$ if $I\mathcal{\subseteq I}_{0}$
and 0 otherwise. The function $c_{1}$ is not necessarily a valid
fractional coloring, but we can turn it into a fractional coloring
$c_{2}$, $c_{2}(I):=\frac{c_{1}(I)}{1-\delta}$. The fact that $c_{2}$
is a valid fractional coloring follows from~(\ref{eq:needtomultiply}).
We have 
\[
\sum_{I\in\mathcal{I}(W)}c_{2}(I)=\frac{\sum_{I\in\mathcal{I}(W)}c_{1}(I)}{1-\delta}\leq\frac{\chi_{\mathrm{frac}}(W)+\delta}{1-\delta}\leq\chi_{\mathrm{frac}}(W)+\frac{\epsilon}{2},
\]
by our choice of $\delta$.

Let $b>\frac{2\cdot\left|\mathcal{I}_{0}\right|}{\epsilon}$. We can
define $c$ by rounding up $c_{2}(I)$ to the closest multiple of
$1/b$. This way, $c$ is still a proper coloring, where we have increased
the total sum of weight on independent sets by at most $\frac{\left|\mathcal{I}_{0}\right|}{b}$.
Thus, we obtain 
\[
\sum_{I\in\mathcal{I}(W)}c(I)\leq\sum_{I\in\mathcal{I}(W)}c_{2}(I)+\frac{\left|\mathcal{I}_{0}\right|}{b}\leq\sum_{I\in\mathcal{I}(W)}c_{2}(I)+\frac{\epsilon}{2}\leq\chi_{\mathrm{frac}}(W)+\epsilon.
\]
Now, $c$ is a rational fractional coloring with common denominator
$b$, as required. This finishes the proof of the Claim.
\end{proof}
Let us fix $\epsilon>0$. For each $b$ sufficiently large, we describe
how to transform a fractional coloring $c:\mathcal{I}(W)\rightarrow\{\frac{j}{b}:j=0,1,\ldots,b\}$
from the Claim (for the given $\epsilon$) into a $b$-fold coloring
$p:\Omega\rightarrow{[K] \choose b}$, where $K=b\sum_{I\in\mathcal{I}(W)}c(I)$
(recall that $K$ is an integer). Actually, it is slightly easier
to allow $p$ to map to subsets of $[K]$ of cardinality \emph{at
least} $b$. Any such $p$ can be obviously corrected eventually.
Let us partition $[K]$ into sets $\bigsqcup_{I\in\mathcal{I}(W)}A_{I}$
so that the size of each set $A_{I}$ is $b\cdot c(I)$. Then for
each $x\in\Omega$ we define $p(x):=\bigcup_{I\in\mathcal{I}_{x}(W)}A_{I}$.
The map $p$ is obviously a proper $b$-fold coloring with $\sum_{I\in\mathcal{I}(W)}c(I)+\epsilon\ge\frac{\chi_{b}(W)}{b}$,
as was needed.
\end{proof}

\section{Cliques and fractional cliques\label{sec:Cliques}}

Below, we define the clique number and the fractional clique number
of a graphon.

\subsection{Cliques}

We first introduce the clique number of a graphon.
\begin{defn}
\label{def:cliquenumber}For a graphon $W$, the \emph{clique number}
$\omega(W)$ is defined as 
\[
\omega(W)=\max\left\{ r:t(K_{r},W)>0\right\} .
\]
\end{defn}

Lower semicontinuity of the clique number follows immediately from
the definition. (As mentioned in Section~\ref{subsec:OurContribution},
the sequence $Y_{n}\equiv\frac{1}{n}\rightarrow Y\equiv0$ shows that
we do not have upper semicontinuity in general.)
\begin{prop}
\label{prop:cliquesemicont}For every sequence of graphons $\left(W_{n}\right)_{n}$
that converges to $W$ in the cut-distance we have $\omega(W)\le\liminf_{n\rightarrow\infty}\omega(W_{n})$.
\end{prop}

\begin{proof}
If $\omega(W)\ge r$, then $t(K_{r},W)>0$. Since subgraph densities
are continuous in the cut-distance, we have $t(K_{r},W_{n})>0$ for
almost all $n$. In particular, $\omega(W_{n})\ge r$ for almost all
$n$.
\end{proof}

\subsection{Fractional cliques}

In analogy with the finite counterpart, we make the following definition.
\begin{defn}
Suppose that $W:\Omega\times\Omega\rightarrow[0,1]$ is a graphon.
We say that a measurable function $f:\Omega\rightarrow[0,+\infty)$
is a \emph{fractional clique} if for every $I\in\mathcal{I}(W)$ we
have $\int_{I}f\leq1$. The \emph{size} of a fractional clique $f$
is $\|f\|:=\int f$. We define the\emph{ fractional clique number}
of $W$ as 
\[
\omega_{\mathrm{frac}}(W)=\sup\int_{\Omega}f,
\]
where the supremum is taken over all fractional cliques in $W$.
\end{defn}

In analogy with Proposition~\ref{prop:chiFRACfiniteVSgraphon}, we
first prove that the fractional clique number was introduced in a
way that is consistent with the graph version.
\begin{prop}
\label{prop:f-cliqueRepresentation}Suppose $G$ is a finite graph
and let $W$ be its graphon representation. Then $\omega_{\mathrm{frac}}(G)=\omega_{\mathrm{frac}}(W)$.
\end{prop}

\begin{proof}
Let $\Omega=\Omega_{1}\sqcup\dots\sqcup\Omega_{n}$ be the partition
for $W$ corresponding to the vertices of~$G$. For any fractional
clique $x:V(G)\rightarrow[0,\infty)$ of the graph $G$ define a function
$f:\Omega\rightarrow[0,\infty)$ with constant value $x(v)n$ in the
interval $\Omega_{v}$ for each $v\in V(G)$. There is a map $\pi:\mathcal{I}(G)\rightarrow\mathcal{I}(W)$
which maps independent sets in $\mathcal{I}(G)$ to corresponding
independent sets in $\mathcal{I}(W)$. Let us verify that $f$ is
a fractional clique for $W$. Let $I\in\mathcal{I}(W)$ be arbitrary.
Let $J\subset V(G)$ consist of the vertices $v$ for which $I\cap\Omega_{v}$
has positive measure. Then $J$ is an independent set in $G$. We
then have
\[
\int_{I}f\le\sum_{v\in J}\int_{\Omega_{v}}f=\sum_{v\in J}x(v)\le1\;,
\]
where the last inequality follows uses that $x$ in a fractional clique
in $G$. Obviously, for the size of $f$ we have $\int_{\Omega}f=\sum_{v\in V(G)}x(v)$.
Therefore, $\omega_{\mathrm{frac}}(G)\leq\omega_{\mathrm{frac}}(W)$.

Now suppose $f:\Omega\rightarrow[0,\infty)$ is any fractional clique
of $W$. Define a function $x:V(G)\rightarrow[0,\infty)$, where $x(v)=\int_{\Omega_{v}}f$.
For each independent set of $J\in\mathcal{I}(G)$, we have
\[
\sum_{v\in J}x(v)=\sum_{v\in J}\int_{\Omega_{v}}f=\int_{\pi(J)}f\leq1\;.
\]
Thus, $x$ is a valid fractional clique for $G$ with size 
\[
\sum_{v\in V(G)}x(v)=\sum_{v\in V(G)}\int_{\Omega_{v}}f=\int_{\Omega}f.
\]
That implies $\omega_{\mathrm{frac}}(G)\geq\omega_{\mathrm{frac}}(W_{G})$,
and finishes the proof.
\end{proof}
We now prove lower semicontinuity of the fractional clique number.
\begin{thm}
\label{thm:liminfFracClique}Let $\left(W_{n}\right)_{n}$ be a sequence
of graphons converging to $W$ in the cut-norm. Then

\[
\omega_{\mathrm{frac}}(W)\leq\liminf_{n\rightarrow\infty}\omega_{\mathrm{frac}}(W_{n})\;.
\]
\end{thm}

\begin{proof}
Let $f:\Omega\rightarrow[0,\infty)$ be a fractional clique of $W$.
We will show that for each graphon $W_{n}$ there is a fractional
clique $f_{n}$ such that $\liminf_{n\rightarrow\infty}\int_{\Omega}f_{n}\geq\int_{\Omega}f$.

For each $n\in\mathbb{N}$, define
\[
s_{n}=\sup_{I\in\mathcal{I}(W_{n})}\int_{I}f\;.
\]
We claim that $\limsup_{n\rightarrow\infty}s_{n}\leq1$. Suppose by
contradiction there exists $\delta>0$ such that $\limsup_{n\rightarrow\infty}s_{n}>1+\delta$.
Fix a subsequence $\left(s_{n_{i}}\right)_{i\in M}$ with $s_{n_{i}}>1+\delta$
for each $i\in M$. Thus, for each $i\in M$ we can find $I_{i}\in\mathcal{I}(W_{n_{i}})$
such that $\int_{I_{i}}f>1+\delta$. By the sequential Banach\textendash Alaoglu
Theorem, the sequence of indicator functions $\boldsymbol{1}_{I_{i}}\in\mathcal{L}^{\infty}(\Omega)$
has a weak{*} accumulation point, say $g\in\mathcal{L}^{\infty}(\Omega)$.
Note that $g$ is bounded from above by~$1$. Define $I=\mathrm{supp}\left(g\right)$.
By Corollary~\ref{cor:independentSetsConvergence}, $I$ is an independent
set of $W$. We have

\[
1+\delta\leq\liminf_{i\in M}\int_{I_{i}}f=\liminf_{i\in M}\int_{\Omega}f\cdot\boldsymbol{1}_{I_{i}}\le\int_{\Omega}f\cdot g\leq\int_{I}f\;,
\]
which contradicts the fact that $f$ is a fractional clique and $I$
is and independent set for $W$.

Since $\limsup_{n\rightarrow\infty}s_{n}\leq1$, we can define for
each $W_{n}$ a valid fractional clique $f_{n}:\Omega\rightarrow[0,\infty)$
for $W_{n}$ by $f_{n}=\frac{f}{s_{n}}$. Furthermore, $\left(f_{n}\right)_{n}$
satisfy $\liminf_{n\rightarrow\infty}\int_{\Omega}f_{n}\geq\int_{\Omega}f$,
as required.
\end{proof}
The next relations between independence number, fractional clique
number, and chromatic number are expected from what we know in the
graph world and this elementary proof is a straightforward adaptation
from the discrete case.
\begin{thm}
\label{thm:ineqFracCliqueIndependent}For any graphon $W:\Omega\times\Omega\rightarrow[0,1]$
we have $1\leq\omega_{\mathrm{frac}}(W)\alpha(W)$.
\end{thm}

\begin{proof}
Define a function $f:\Omega\rightarrow[0,\infty)$ by $f(x)=1/\alpha(W)$.
Clearly, $f$ is a valid fractional clique, since for every independent
set $I\in\mathcal{I}(W)$ we have$\int_{I}f=\frac{\nu(I)}{\alpha(W)}\leq1.$
Thus, $\omega_{\mathrm{frac}}(W)\geq\int f=1/\alpha(W)$ and the claim
follows.
\end{proof}
\begin{thm}
\label{cor:chromaticAndIndependent}For any graphon $W:\Omega\times\Omega\rightarrow[0,1]$we
have $1\leq\chi(W)\alpha(W)$.
\end{thm}

\begin{proof}
Consider a decomposition of $\Omega$ into independent sets $I_{1}\sqcup I_{2}\sqcup\ldots\sqcup I_{\chi(W)}=\Omega\mod0$.
Now, the inequality follows from

\[
1=\sum_{j=1}^{\chi(W)}\int_{I_{j}}1\leq\sum_{j=1}^{\chi(W)}\alpha(W)=\chi(W)\alpha(W).
\]
\end{proof}
Alternatively, the proof of the previous result would follow from
Theorem \ref{thm:ineqFracCliqueIndependent} if we would have the
inequality $\omega_{\mathrm{frac}}(W)\leq\chi_{\mathrm{frac}}(W)$.
This inequality will only be obtained in Section \ref{sec:Duality},
where we investigate duality properties between these parameters.
That is why at this point we decided to get it from the same argument
given for graphs.

However, certain properties can be more challenging to obtain. For
instance, from what we know about fractional cliques in graphs it
is not surprising that if a graph $G$ satisfies $t(G,W)>0$, then
$\omega_{\mathrm{frac}}(G)\leq\omega_{\mathrm{frac}}(W)$. In the
discrete world, this is a direct consequence of the definition. However,
for graphons this statement requires a technical proof. 
\begin{prop}
\label{prop:subgraphFracClique}Suppose that $G$ is a graph and $W$
is a graphon with $t(G,W)>0$. Then $\omega_{\mathrm{frac}}(G)\leq\omega_{\mathrm{frac}}(W)$.
\end{prop}

\begin{proof}
Let $\epsilon>0$ and let $g:V(G)\rightarrow[0,1]$ be a fractional
clique of $G$. To prove the statement we shall find a fractional
clique $f:\Omega\rightarrow[0,\infty)$ for $W$ such that 
\[
\int_{\Omega}f\geq\left(1-\epsilon\right)\sum_{v\in V(G)}g(v).
\]
Let $\delta=\left(\frac{\epsilon}{\left|V(G)\right|}\right)^{2}$.
Since $t(G,W)>0$, by Lemma~\ref{lem:rectangles}, we can find $\beta>0$
disjoint sets $A_{v}\subseteq\Omega$, $v\in V(G)$, with positive
measure $\nu(A_{v})=\beta$ such that for each $uv\in E(G)$,
\begin{equation}
\mu^{\otimes2}(A_{v}\times A_{u}\setminus\mathrm{supp}\left(W\right))<\delta\beta^{2}.\label{eq:(3)}
\end{equation}

Define
\[
f(x):=\begin{cases}
\frac{\left(1-\epsilon\right)}{\beta}g(v) & \text{for }x\in A_{v}\\
0 & \text{otherwise}.
\end{cases}
\]
Thus, we obtain 
\[
\int_{\Omega}f=\left(1-\epsilon\right)\sum_{v\in V(G)}g(v).
\]
It remains to prove that for each $I\in\mathcal{I}(W)$ we have $\int_{I}f\leq1$.
It follows from~(\ref{eq:(3)}) that for each $uv\in E(G)$ we have
\begin{equation}
\mu(I\cap A_{v})\leq\sqrt{\delta}\beta\text{ or }\mu(I\cap A_{u})\leq\sqrt{\delta}\beta.\label{eq:claim}
\end{equation}
Let $J\subset V(G)$ be the set of vertices for which $\mu(I\cap A_{v})>\sqrt{\delta}\beta$.
By~(\ref{eq:claim}) it follows that $J$ is an independent set.
Therefore, $\sum_{v\in J}g(v)\leq1$. We have
\begin{equation}
\int_{I}f\leq\sum_{v\in J}\int_{A_{v}}f+\int_{I\setminus\cup_{v\in J}A_{v}}f\;.\label{eq:eq1}
\end{equation}
The first term can be bounded by 
\begin{equation}
\sum_{v\in J}\int_{A_{v}}f\leq\sum_{v\in J}\beta\frac{\left(1-\epsilon\right)}{\beta}g(v)\leq\left(1-\epsilon\right)\sum_{v\in J}g(v)\leq1-\epsilon.\label{eq:eq2}
\end{equation}
For the second term, we have
\[
\int_{I\setminus\cup_{v\in J}A_{v}}f=\sum_{v\in V(G)\setminus J}\int_{I\cap A_{v}}f.
\]
For each $v\in V(G)\setminus J$, we have
\begin{align*}
\int_{I\cap A_{v}}f & \leq\sqrt{\delta}\beta\frac{\left(1-\epsilon\right)}{\beta}g(v)\\
 & =\sqrt{\delta}\left(1-\epsilon\right)g(v)\\
 & \leq\frac{\epsilon}{V(G)}\left(1-\epsilon\right)g(v)\\
 & \leq\frac{\epsilon}{\left|V(G)\right|}.
\end{align*}
Therefore, 
\begin{equation}
\int_{I\setminus\cup_{v\in J}A_{v}}f\leq\sum_{v\in V(G)\setminus J}\frac{\epsilon}{\left|V(G)\right|}\leq\epsilon.\label{eq:eq3}
\end{equation}
By (\ref{eq:eq1}), (\ref{eq:eq2}), and (\ref{eq:eq3}), it follows
that $\int_{I}f\leq1$ and therefore $f$ is a valid fractional clique
for $W$. This finishes the proof.
\end{proof}
The results above allow us to express the fractional clique number
of a graphon using the same parameter of finite graphs appearing in
it.
\begin{cor}
\label{cor:fraccliqSUPfinite}Suppose that $W$ is a graphon. Then
$\omega_{\mathrm{frac}}(W)=\sup_{G:t(G,W)>0}\omega_{\mathrm{frac}}(G)$.
\end{cor}

\begin{proof}
Proposition~\ref{prop:subgraphFracClique} gives $\omega_{\mathrm{frac}}(W)\ge\sup_{G:t(G,W)>0}\omega_{\mathrm{frac}}(G)$.
On the other hand, taking $G_{n}\sim\mathbb{G}(n,W)$, we get a sequence
$(G_{n})$ of graphs which all satisfy (almost surely) that $t(G_{n},W)>0$.
Considering arbitrary graphon representations $W_{n}$ of these graphs
$G_{n}$, we know that $W_{n}$ converge to $W$ in the cut-distance.
Thus Theorem~\ref{thm:liminfFracClique} tells us that $\omega_{\mathrm{frac}}(W)\leq\liminf_{n\rightarrow\infty}\omega_{\mathrm{frac}}(W_{n})$.
As $\omega_{\mathrm{frac}}(W_{n})=\omega_{\mathrm{frac}}(G_{n})$
by Proposition~\ref{prop:f-cliqueRepresentation}, we also prove
the inequality $\omega_{\mathrm{frac}}(W)\le\sup_{G:t(G,W)>0}\omega_{\mathrm{frac}}(G)$.
\end{proof}
Last, we relate the fractional clique number and the integral clique
number of a graphon.
\begin{prop}
Suppose that $W$ is a graphon. Then $\omega_{\mathrm{frac}}(W)\ge\omega(W)$.
\end{prop}

\begin{proof}
By Corollary~\ref{cor:fraccliqSUPfinite}, we have $\omega_{\mathrm{frac}}(W)=\sup_{G:t(G,W)>0}\omega_{\mathrm{frac}}(G)$.
Further, Definition~\ref{def:cliquenumber} gives $\omega(W)=\sup_{G:t(G,W)>0}\omega(G)$.
We thus get the statement combining these two relations together with
the fact that $\omega_{\mathrm{frac}}(G)\ge\omega(G)$ for each finite
graph.
\end{proof}

\section{\label{sec:Duality}Duality between fractional cliques and fractional
coloring\label{sec:DualityCliqueColoring}}

The classical LP duality states that for a finite graph $G$ we have
that 
\begin{equation}
\omega_{\mathrm{frac}}(G)=\chi_{\mathrm{frac}}(G)\;.\label{eq:LPfiniteOmegaChi}
\end{equation}
Note that such a relation cannot hold for graphons in general. Indeed,
taking the graphon $W$ constructed in Example~\ref{ex:leader},
we get that
\[
\omega_{\mathrm{frac}}(W)=\sup_{G:t(G,W)>0}\omega_{\mathrm{frac}}(G)\overset{\eqref{eq:LPfiniteOmegaChi}}{=}\sup_{G:t(G,W)>0}\chi_{\mathrm{frac}}(G)<\chi_{\mathrm{frac}}(W)\;.
\]

So, in this section we establish the maximum that can be established
in this situation: weak LP duality and complementary slackness.
\begin{thm}[Weak LP duality]
 \label{thm:weakduality}For every graphon $W:\Omega\times\Omega\rightarrow[0,1]$
we have
\[
\omega_{\mathrm{frac}}(W)\leq\chi_{\mathrm{frac}}(W).
\]
\end{thm}

\begin{proof}
Suppose $f:\Omega\rightarrow[0,1]$ is an arbitrary fractional clique
of $W$ and $c:\mathcal{I}(W)\rightarrow[0,1]$ is an arbitrary fractional
coloring of $W$. Since $\sum_{I\in\mathcal{I}_{x}(W)}c(I)\geq1$
for almost every $x\in\Omega$, we can write 
\begin{equation}
\int_{x\in\Omega}f(x)\leq\int_{x\in\Omega}\left(f(x)\sum_{I\in\mathcal{I}_{x}(W)}c(I)\right)=\sum_{I\in\mathcal{I}(W)}\left(c(I)\int_{x\in I}f(x)\right)\leq\sum_{I\in\mathcal{I}(W)}c(I),\label{eq:reference}
\end{equation}
since $\int_{I}f\leq1$ for each $I\in\mathcal{I}(W)$. This finishes
the proof.
\end{proof}
\begin{thm}[Complementary slackness]
Let $W:\Omega\times\Omega\rightarrow[0,1]$ be a graphon with a fractional
clique $f:\Omega\rightarrow[0,1]$ and a fractional coloring $c:\mathcal{I}(W)\rightarrow[0,1]$.
We have
\begin{equation}
\int_{\Omega}f=\sum_{I\in\mathcal{I}(W)}c(I)\label{eq:strongduality}
\end{equation}
if and only if
\begin{equation}
f(x)\left(\sum_{I\in\mathcal{I}_{x}(W)}c(I)-1\right)=0\label{eq:referenceequality}
\end{equation}
for almost every \textup{$x\in\Omega$} and 
\begin{equation}
c(I)\left(\int_{I}f-1\right)=0\label{eq:referenceequality2}
\end{equation}
for every $I\in\mathcal{I}(W)$. In this case, for every $A\subset\Omega$
with positive measure we have $f\mid_{A}=0$ almost everywhere if
and only if 
\[
\sum_{I\in\mathcal{I}_{x}(W)}c(I)>1
\]
 for almost every \textup{$x\in A$.} 
\end{thm}

\begin{proof}
Notice that~(\ref{eq:strongduality}) holds if and only if both inequalities
in~(\ref{eq:reference}) are at equality. Since $f$ and $c$ are
nonnegative, (\ref{eq:referenceequality}) and~(\ref{eq:referenceequality2})
follow. This proves the first part. To get second part , it is enough
to recall that$\sum_{I\in\mathcal{I}_{x}(W)}c(I)\geq1$.
\end{proof}

\section{Perfect graphons\label{sec:PerfectGraphons}}

The notion of perfect graph is central in combinatorial optimization.
Recall that a graph $G$ is \emph{perfect} if for every induced subgraph
$H$ of $G$, we have $\chi(H)=\omega(H)$. In the remarkable work
of Chudnovsky, Robertson, Seymour, and Thomas~\cite{MR2233847},
a forbidden subgraph characterization for perfect graphs is given.
The result settles a problem that for four decades was known as the
strong perfect graph conjecture. It can be stated as follows.
\begin{thm}[Strong perfect graph Theorem]
A graph is perfect if and only if it contains no induced odd cycle
of length more than 4 and no complement of an induced of cycle of
length more than 4.
\end{thm}

It is not clear what should be the right definition of perfect graphons.
We offer two definitions. The first definition is in terms of subgraph
densities.
\begin{defn}
We call a graphon $W:\Omega\times\Omega\rightarrow[0,1]$ \emph{subgraph-perfect}
if for every finite graph $H$ with $t_{\mathrm{ind}}(H,W)>0$ we
have $\chi(H)=\omega(H)$. Further, we call $W$ \emph{properly subgraph-perfect}
if it is subgraph-perfect and $\chi(W)<\infty$.
\end{defn}

The other definition of perfect graphons mimics more closely the graph
definition. Note that the property $\chi(G)\ge\omega(G)$ which holds
for finite graphs and was a motivation for the notion of perfect graphons
holds for graphons, too.
\begin{defn}
We call a graphon $W:\Omega\times\Omega\rightarrow[0,1]$ \emph{inheritance-perfect}
if for every set $A\subset\Omega$ of positive measure we have $\chi(W[A])=\omega(W[A])$.
Further, we call $W$ \emph{properly inheritance-perfect} if it is
induced-perfect and $\chi(W)<\infty$.
\end{defn}

We can characterize subgraph-perfect graphons in the same fashion
as perfect graphs by the density of induced odd cycles and its complements.
\begin{prop}
A graphon $W$ is subgraph-perfect if and only if for every odd integer
$\ell\geq5$ we have $t_{\mathrm{ind}}(C_{\ell},W)=0$ and $t_{\mathrm{ind}}(\overline{C_{\ell}},W)=0$.
\end{prop}

\begin{proof}
For every odd integer $\ell\geq5$, we have $\omega(C_{\ell})<\chi(C_{\ell})$
and $\omega(\overline{C_{\ell}})<\chi(\overline{C_{\ell}})$. Thus,
if $W$ is subgraph-perfect, then $t_{\mathrm{ind}}(C_{\ell},W)=0$
and $t_{\mathrm{ind}}(\overline{C_{\ell}},W)=0$.

Now, assume $W$ is not subgraph-perfect. Then, there is a finite
imperfect graph $H$ with $t_{\mathrm{ind}}(H,W)>0$. By the Strong
perfect graph theorem we have that $t_{\mathrm{ind}}(C_{\ell},H)>0$
or $t_{\mathrm{ind}}(\overline{C_{\ell}},H)>0$ for some odd integer
$\ell\geq5$. By Exercise~7.6\footnote{Exercise~7.6 in \cite{Lovasz2012} is stated for the parameter $t(\cdot,\cdot)$
rather than $t_{\mathrm{ind}}(\cdot,\cdot)$. The statement and the
proof hold \emph{mutatis mutandis}.} of \cite{Lovasz2012} we get that $t_{\mathrm{ind}}(C_{\ell},W)>0$
or $t_{\mathrm{ind}}(\overline{C_{\ell}},W)>0$. This finishes the
proof.
\end{proof}
Recall that the \emph{complement} of a graphon $W$ is defined as
$\overline{W}(x,y)=1-W(x,y)$ for every $x,y\in\Omega$. Naturally,
we obtain the weak version of the last theorem.
\begin{cor}
A graphon is subgraph-perfect if and only if its complement is subgraph-perfect.
\end{cor}

We now relate subgraph-perfectness and induced-perfectness.
\begin{thm}
\label{thm:graphonIsPerfect}Let $W:\Omega\times\Omega\rightarrow[0,1]$
be a properly subgraph-perfect graphon. Then $W$ is inheritance-perfect.
\end{thm}

\begin{proof}
We shall first prove that $\omega(W)=\chi(W)$. To this end, consider
the sequence in the statement of Theorem~\ref{thm:chromaticLimit}
(b). Let $G_{n}$ be the graphs in that sequence and let $W_{n}$
be their graphon representation. Now, we have

\begin{equation}
\omega(W)=\lim_{n\rightarrow\infty}\omega(G_{n}).\label{eq:lim1}
\end{equation}
Furthermore, by Theorem~\ref{thm:chromaticLimit} we almost surely
have

\begin{equation}
\chi(W)=\lim_{n\rightarrow\infty}\chi(G_{n}).\label{eq:lim2}
\end{equation}

Since $W$ is properly subgraph-perfect, the graphs $G_{n}$ are all
perfect (almost surely). Thus, $\omega(G_{n})=\chi(G_{n})$, and $\omega(W)=\chi(W)$
follows from~(\ref{eq:lim1}) and~(\ref{eq:lim2}).

Now, observe that induced-perfectness follows. Indeed, if $W$ is
subgraph-perfect and we are given a set $A\subset\Omega$ of positive
measure, we can run the above argument for the graphon $W[A]$ (which
is obviously subgraph-perfect) to conclude that $\chi(W[A])=\omega(W[A])$.
\end{proof}
We cannot reverse Theorem~\ref{thm:graphonIsPerfect} directly. For
example, consider $W$ to be a graphon representation of a triangle
$K_{3}$. Now, take $U:=\frac{1}{2}W$, that is, $U$ is a graphon
representation of a triangle, where the edges are represented by value
$\frac{1}{2}$. Obviously, $U$ is inheritance-perfect but we have
$t_{\mathrm{ind}}(C_{5},W)>0$. However, we do not know whether we
can reverse Theorem~\ref{thm:graphonIsPerfect} for $\{0,1\}$-valued
graphons. We pose this as an open problem.
\begin{problem}
\label{prob:reverseperfect}Suppose that $W$ is a $\{0,1\}$-valued
inheritance-perfect graphon. Is $W$ also subgraph-perfect?
\end{problem}

Our last open problem is the ``weak perfect graph theorem'' for
inheritance-perfect graphons. Note that a positive answer to this
problem would be automatically implied by a positive answer to Problem~\ref{prob:reverseperfect}.
\begin{problem}
Suppose that $W$ is a $\{0,1\}$-valued inheritance-perfect graphon.
Is its complement $\overline{W}$ also inheritance-perfect?
\end{problem}

\bibliographystyle{plain}
\bibliography{bibl}

\end{document}